\documentclass[12pt]{article}

\usepackage{amsmath,amsthm,amssymb,mathrsfs}
\usepackage[dvips]{graphicx}
\usepackage[dvips]{epsfig}
\usepackage{color}

\def\titlerunning#1{\gdef\titrun{#1}}
\makeatletter
\def\author#1{\gdef\autrun{\def\and{\unskip, }#1}\gdef\@author{#1}}
\def\address#1{{\def\and{\\\hspace*{18pt}}\renewcommand{\thefootnote}{}%
\footnote {#1}}%
\markboth{\autrun}{\titrun}}
\makeatother
\def\email#1{e-mail: #1}

\def\keywords#1{\par\medskip
\noindent\textbf{Keywords.} #1}

\def\b{{\bf b}}

\def\e{{\bf e}}
\def\f{{\bf f}}

\def\i{{\bf i}}
\def\j{{\bf j}}
\def\k{{\bf k}}

\def\r{{\bf r}}
\def\s{{\bf s}}

\def\u{{\bf u}}

\def\Q{{\bf Q}}

\def\rd{{\rm d}}

\def\RR{{\mathop{{\rm I}\kern-.2em{\rm R}}\nolimits}}

\newcommand{\bfomega}{{\mbox{\boldmath $\omega$}}}

\newcommand{\X}{{\cal X}}
\newcommand{\Y}{{\cal Y}}

\renewcommand{\Q}{{\cal Q}}

\newcommand{\A}{{\cal A}}
\newcommand{\B}{{\cal B}}
\newcommand{\C}{{\cal C}}
\newcommand{\nn}{\mathbb{N}}
\newcommand{\rr}{\mathbb{R}}
\newcommand{\cc}{\mathbb{C}}
\newcommand{\hh}{\mathbb{H}}

\newcommand{\be}{\begin{equation}}
\newcommand{\ee}{\end{equation}}
\newcommand{\ba}{\begin{eqnarray}}
\newcommand{\ea}{\end{eqnarray}}
\newcommand{\bi}{\begin{itemize}}
\newcommand{\ei}{\end{itemize}}

\newtheorem{teo}{Theorem}[section]
\newtheorem{cor}[teo]{Corollary}
\newtheorem{oss}[teo]{Remark}
\newtheorem{defi}[teo]{Definition}
\newtheorem{lem}[teo]{Lemma}
\newtheorem{pro}[teo]{Proposition}
\newtheorem{ese}[teo]{Example}

\begin{document}

\titlerunning{}

\title{\bf 
A comprehensive characterization \\ 
of the set of polynomial curves with \\
rational rotation-minimizing frames
}

\author{Rida~T.~Farouki, Graziano Gentili, Carlotta Giannelli, \\ 
Alessandra Sestini, and Caterina Stoppato}

\date{}

\maketitle

\address{Rida~T.~Farouki,
Department of Mechanical and Aerospace Engineering,
University of California, Davis, CA 95616, USA, corresponding author:
e--mail farouki@ucdavis.edu, phone 530--752-1779, fax 530--752--4158
\and
Graziano Gentili, Carlotta Giannelli, Alessandra Sestini, Caterina Stoppato,
Dipartimento di Matematica e Informatica ``U. Dini,'' 
Universit\`a %degli Studi 
di Firenze, Viale Morgagni 67/A, I--50134 Firenze, Italy, \hskip7pt\email{gentili@math.unifi.it, carlotta.giannelli@unifi.it, alessandra.sestini@unifi.it, stoppato@math.unifi.it}
}

\begin{abstract}
\noindent
A rotation--minimizing frame $(\f_1,\f_2,\f_3)$ on a space curve $\r(\xi)$ 
defines an orthonormal basis for $\rr^3$ in which $\f_1=\r'/|\r'|$ is the 
curve tangent, and the normal--plane vectors $\f_2$, $\f_3$ exhibit no
instantaneous rotation about $\f_1$. Polynomial curves that admit \emph
{rational} rotation--minimizing frames (or RRMF curves) form a subset 
of the Pythagorean--hodograph (PH) curves, specified by integrating the 
form $\r'(\xi)=\A(\xi)\,\i\,\A^*(\xi)$ for some quaternion polynomial 
$\A(\xi)$. By introducing the notion of the \emph{rotation indicatrix} and 
the \emph{core} of the quaternion polynomial $\A(\xi)$, a comprehensive 
characterization of the complete space of RRMF curves is developed, that 
subsumes all previously known special cases. This novel characterization
helps clarify the structure of the complete space of RRMF curves, 
distinguishes the spatial RRMF curves from trivial (planar) cases,
and paves the way toward new construction algorithms.
\end{abstract}

%%%%%%%%%%%%%%%%%%%%%%%%%%%%%%%%%%%%%%%%%%%%%%%%%%%%%%

\newpage

\bigskip\noindent
{\bf Short title}: rational rotation--minimizing frames

12D05, 12Y05, 14H45, 14H50, 53A04, 68U05, 68U07

\bigskip\noindent


\newpage

\section{Introduction}

The specification of rigid--body motions involving coordinated 
translational and orientational components is a fundamental problem 
in spatial kinematics, of relevance to applications such as robot path 
planning, computer animation, motion control, and geometric design. 
Among all conceivable correlations of position and orientation along 
a specified path, perhaps the most important and intuitive is the \emph
{adapted rotation--minimizing motion}, in which the body exhibits 
no instantaneous rotation about the path tangent --- i.e., its angular 
velocity component in the tangent direction is exactly zero.

An orthonormal frame exhibiting this property along a parametric curve 
$\r(\xi)$ in $\mathbb{R}^3$ is known as a \emph{rotation--minimizing 
frame} (RMF) or \emph{Bishop frame} \cite{bishop75}. However, the RMFs 
on polynomial or rational curves do not in general admit simple (rational) 
closed--form expression, and must be approximated --- see, for example
\cite{farouki03,wang97,wang08}. Exact representations are clearly preferable 
whenever possible, not only because they avoid approximation errors, but 
also because they are more concise and ``robust.'' Such considerations have 
prompted great interest in the study of polynomial curves $\r(\xi)$ with 
RMFs that admit a \emph{rational} dependence on the curve parameter $\xi$, 
and considerable progress has recently been achieved in the characterization 
and construction of such \emph{rational rotation--minimizing frame} (RRMF) 
\emph{curves} --- especially the simplest non--trivial examples, the 
quintics \cite{farouki10a,farouki09,farouki10b,farouki12b,han08}.

However, the different types of RRMF curves studied thus far have been
investigated on a case--by--case basis, through idiosyncratic approaches, 
and these known cases suggest a rich structure to the entire set of RRMF 
curves. A theoretical framework that encompasses all the currently--known 
RRMF curve types, illuminates the structure of the entire space of RRMF 
curves, and furnishes algorithms for their construction through the
satisfaction of geometrical constraints, is therefore highly desirable.

To ensure a rational unit tangent vector, this problem must be addressed in the 
established theoretical framework of the spatial \emph{Pythagorean--hodograph} 
(PH) \emph{curves} \cite{farouki08}, i.e., polynomial curves $\r(\xi)=
(x(\xi),y(\xi),z(\xi))$ in $\mathbb{R}^3$ such that the components of the 
derivative or hodograph $\r'(\xi)=(x'(\xi),y'(\xi),z'(\xi))$ satisfy 
\be
\label{pythag}
x'^2(\xi) + y'^2(\xi) + z'^2(\xi) \,\equiv\, \sigma^2(\xi)
\ee
for some polynomial $\sigma(\xi)$. The solutions to (\ref{pythag}) can be 
characterized \cite{choi02b,farouki02} in terms of the quaternion algebra 
$\hh=\rr+\rr\i+\rr\j+\rr\k$. We identify with $\rr^3$ the vector subspace 
$\rr\i+\rr\j+\rr\k\subset\hh$, whose elements are called \emph{pure vectors}. 
When the Euclidean norm $|\A|$ of $\A \in \hh$ equals $1$, $\A$ is called a 
\emph{unit quaternion}. A pure vector $\u \in \rr\i+\rr\j+\rr\k$ with $|\u|=1$ 
is called a \emph{unit vector}. Now, for some quaternion polynomial
\be
\label{A}
\A(\xi) \,=\, u(\xi) + v(\xi)\,\i + p(\xi)\,\j + q(\xi)\,\k \,, 
\ee
where $u(\xi),v(\xi),p(\xi),q(\xi)$ are real polynomials, to satisfy (\ref
{pythag}) the hodograph $\r'(\xi)$ must be of the form
\ba
\label{hodo}
\r'(\xi) \!\! &=& \!\! \A(\xi)\,\i\,\A^*(\xi) \,=\,
[\,u^2(\xi)+v^2(\xi)-p^2(\xi)-q^2(\xi)\,]\,\i  \nonumber \\
&+& \!\! 2\,[\,u(\xi)q(\xi)+v(\xi)p(\xi)\,]\,\j  
\,+\, 2\,[\,v(\xi)q(\xi)-u(\xi)p(\xi)\,]\,\k \,, 
\ea
$\A^*(\xi)=u(\xi)-v(\xi)\,\i-p(\xi)\,\j-q(\xi)\,\k$ being the 
conjugate of $\A(\xi)$. Since this amounts to specifying $\r'(\xi)$ 
through a continuous family of scaling/rotation transformations acting 
on the unit vector\footnote{The choice of $\i$ is merely conventional:
it may be replaced by any other unit vector.} $\i$, the quaternion 
polynomial $\A(\xi)$ is said to \emph{generate} (or be the \emph
{pre--image} of) the hodograph $\r'(\xi)$.

The present paper re--interprets the characterization of RRMF curves 
due to Han \cite{han08} in terms of the natural Euclidean metric of the 
quaterion space $\hh$. After reviewing some basic properties of RRMF 
curves in Section~\ref{sec:pre}, the notion of the \emph{rotation indicatrix}
of a quaternion polynomial $\A(\xi)$ is introduced in Section~\ref
{sec:fourierindicatrix}. The characterization of the class $\mathscr{F}$ 
of quaternion polynomials that generate RRMF curves is then reduced to 
the study of the class $\mathscr{F}_0$ of quaternion polynomials with 
vanishing rotation indicatrix in Section~\ref{sec:reduction}, and two
characterizations of $\mathscr{F}_0$ are presented in Section~\ref
{sec:vanishingindicatrix}.

Based on these results, a precise characterization for the set of 
quaternion polynomials that generate non--planar polynomial PH curves 
with rational RMFs is developed in Section~\ref{sec:nonplanar}. Moreover, 
examples of polynomials $\A(\xi)\in\mathscr{F}_0$ of degree $n$ that 
generate non--planar RRMF curves are exhibited for all $n\geq3$. In 
particular, complete characterizations of such polynomials are stated 
for $n=3$ and $4$ in Section~\ref{sec:nontrivial}. Finally, Section~\ref
{sec:exm} presents a selection of example curves generated by polynomials
$\A(\xi)\in\mathscr{F}\setminus\mathscr{F}_0$, and Section~\ref{sec:closure} 
summarizes the results of this study and makes some concluding remarks.

\section{Preliminaries on RRMF curves}\label{sec:pre}

For a PH curve $\r(\xi)$ satisfying \eqref{hodo}, the \emph{parametric 
speed} (i.e., the derivative $\rd s/\rd\xi$ of arc length $s$ with respect 
to the curve parameter $\xi$) is defined by
\[
\sigma(\xi) \,=\, |\r'(\xi)| \,=\, |\A(\xi)|^2 
\,=\, u^2(\xi)+v^2(\xi)+p^2(\xi)+q^2(\xi) \,.
\]
Since $\sigma(\xi)$ is a polynomial, PH curves possess rational unit tangent 
vectors, polynomial arc length functions, and many other advantageous features 
\cite{farouki08}. 

Before proceeding, we introduce some useful notations. For any field $F$, 
the symbol $F[\xi]$ will denote the ring of polynomials over $F$ in the 
variable $\xi$. For any choice of polynomials $p_1(\xi),\ldots,p_n(\xi)\in 
F[\xi]$, we will denote by $\gcd_F(p_1(\xi),\ldots,p_n(\xi))$ their monic 
\emph{greatest common divisor}.

\begin{defi}\label{primitive}
A hodograph $\r'(\xi)=(x'(\xi),y'(\xi),z'(\xi))$ is \emph{primitive} if 
its components are coprime in $\rr[\xi]$, i.e., $\gcd_\rr(x'(\xi),y'(\xi),
z'(\xi))=1$.
\end{defi}

Clearly, for $\r'(\xi)$ to be primitive, the components $u(\xi),v(\xi),
p(\xi),q(\xi)$ of (\ref{A}) must be coprime in $\rr[\xi]$. This is a 
necessary, but not sufficient, condition. Writing (\ref{A}) in terms of 
the complex polynomials $\alpha(\xi)=u(\xi)+v(\xi)\,\i$, $\beta(\xi)=
p(\xi)+q(\xi)\,\i$ as $\A(\xi)=\alpha(\xi)+\beta(\xi)\,\j$, one can verify 
\cite{farouki04} that $x'(\xi)$, $y'(\xi)$, $z'(\xi)$ have the common 
factor $|\gcd_\cc(\alpha(\xi),\beta^*(\xi))|^2$ where $\beta^*(\xi)=p(\xi)
-q(\xi)\,\i$ is the conjugate of $\beta(\xi)$. Hence, $\alpha(\xi)$ and 
$\beta^*(\xi)$ must also be coprime in order for the expression (\ref
{hodo}) to generate a primitive hodograph.

The {\it Euler--Rodrigues frame\/} (ERF) is a rational orthonormal frame
for $\rr^3$, defined \cite{choi02a} on any spatial PH curve by
\be
\label{erf}
(\e_1(\xi),\e_2(\xi),\e_3(\xi)) \,=\, \frac{(\A(\xi)\,\i\,\A^*(\xi),
\A(\xi)\,\j\,\A^*(\xi),\A(\xi)\,\k\,\A^*(\xi))}{|\A(\xi)|^2} \,.
\ee
This is an ``adapted'' frame, in the sense that $\e_1$ coincides with the 
curve tangent, while $\e_2$ and $\e_3$ span the curve normal plane at each 
point. The ERF variation is characterized by its angular velocity $\bfomega
=\omega_1\e_1+\omega_2\e_2+\omega_3\e_3$ through the relations $\e_k'=
\bfomega\times\e_k$ for $k=1,2,3$. In particular, the angular velocity 
component $\omega_1$, specified \cite{farouki16} by
\be
\label{omega1}
\omega_1 \,=\, \e_3\cdot\e_2' \,=\, -\,\e_2\cdot\e_3' 
\,=\, \frac{2(uv'-u'v-pq'+p'q)}{u^2+v^2+p^2+q^2} \,,
\ee
represents the rate of rotation of $\e_2$ and $\e_3$ about $\e_1$. Among all 
possible orthonormal adapted frames, the \emph{rotation--minimizing frames} 
(RMFs) that satisfy $\omega_1 \equiv 0$, also known \cite{bishop75} as 
\emph{Bishop frames}, are of the greatest interest in various practical
applications, such as spatial motion planning, computer animation, robotics, 
and swept surface constructions.

If $(\f_1,\f_2,\f_3)$ is an adapted RMF on $\r(\xi)$, where $\f_1=\r'/|\r'|$ 
is the curve tangent, the condition $\omega_1\equiv0$ implies that $\f_2$ 
and $\f_3$ exhibit no instantaneous rotation about $\f_1$. Note that a 
one--parameter family of RMFs exists on any given curve, since the initial 
normal--plane orientation of $\f_2$, $\f_3$ may be freely chosen. A polynomial 
curve with a rational RMF (called an \emph{RRMF curve}) is necessarily a 
PH curve, since PH curves are the only polynomial curves with rational unit 
tangent vectors.\footnote{Rational curves with rational RMFs also exist 
\cite{barton10}, but are not considered herein.} Although the ERF is a 
rational  adapted frame, it is clear from (\ref{omega1}) that it is not, 
in general, an RMF. Nevertheless, it serves \cite{han08} as a useful 
intermediary in identifying PH curves that admit rational RMFs.

As noted by Han \cite{han08}, the normal--plane vectors $\f_2(\xi),
\f_3(\xi)$ of a rational RMF must be obtainable from the ERF vectors 
$\e_2(\xi),\e_3(\xi)$ through a rational normal--plane rotation, of 
the form
\be
\label{f23}
\left[\! \begin{array}{c} 
\f_2(\xi) \\ \f_3(\xi) 
\end{array} \!\right] \,=\,
\frac{1}{a^2(\xi)+b^2(\xi)}
\left[\! \begin{array}{cc} 
a^2(\xi)-b^2(\xi) & -\,2\,a(\xi)b(\xi) \\
2\,a(\xi)b(\xi) & a^2(\xi)-b^2(\xi) 
\end{array} \!\right]
\left[\! \begin{array}{c} 
\e_2(\xi) \\ \e_3(\xi) 
\end{array} \!\right]
\ee
for coprime real polynomials $a(\xi)$, $b(\xi)$. This amounts to defining
$\f_2(\xi),\f_3 (\xi)$ by a normal--plane rotation of $\e_2(\xi)$, $\e_3
(\xi)$ through the angle 
\[
\theta(\xi) \,=\, -\,2\arctan\frac{b(\xi)}{a(\xi)} \,,
\] 
which incurs angular velocity $\theta'=-\,2(ab'-a'b)/(a^2+b^2)$ in the 
$\f_1$ direction. For $\f_2,\f_3$ to be rotation--minimizing, this must 
exactly cancel the ERF angular velocity component (\ref{omega1}). Based 
on these considerations, Han \cite{han08} stated the following criterion 
identifying the RRMF curves as a subset of all PH curves.

\begin{teo}
\label{Han}
The PH curve generated by the quaternion polynomial (\ref{A}) has a 
rational RMF if and only if coprime real polynomials $a(\xi)$, $b(\xi)$ 
exist, such that the components of $\A(\xi)$ satisfy
\be
\label{rrmfeqn}
\frac{uv'-u'v-pq'+p'q}{u^2+v^2+p^2+q^2}
\,\equiv\, \frac{ab'-a'b}{a^2+b^2} \,.
\ee
\end{teo}

\medskip\noindent
As noted in \cite{farouki12}, if condition~\eqref{rrmfeqn} is satisfied and 
$\B(\xi):=\A(\xi)(a(\xi)-b(\xi)\,\i)$, the rational RMF can be expressed as
\be
\label{rmf}
(\f_1(\xi),\f_2(\xi),\f_3(\xi)) \,=\, \frac{(\B(\xi)\,\i\,\B^*(\xi),
\B(\xi)\,\j\,\B^*(\xi),\B(\xi)\,\k\,\B^*(\xi))}{|\B(\xi)|^2} \,.
\ee

Although important results concerning the identification 
and construction of RRMF curves have recently been derived \cite
{barton10,cheng16,choi02a,farouki10a,farouki09,farouki12,farouki13,farouki10b,farouki12b,han08}
a comprehensive theory of them has thus far remained elusive \cite{farouki16}.
The goal of this study is to develop a unified approach to the set of PH 
curves that satisfy the RRMF condition \eqref{rrmfeqn} --- see Theorems~\ref
{finale} and \ref{G} below. This approach provides a new understanding of 
expression~\eqref{rmf}, expressed in Proposition~\ref{prop:great}. Since the 
Frenet frame of any planar PH curve is trivially a rational RMF, the focus 
is mainly on \emph{spatial} PH curves, with non--vanishing torsion. However, 
the analysis covers all cases and includes a criterion to distinguish
non--planar RRMF curves from planar curves --- see Theorem~\ref{thm:sum}
below. 

\section{Rotation indicatrix of RRMF curves}
\label{sec:fourierindicatrix}

Recall that $\hh$ denotes the real algebra of quaternions, and let 
$\hh[\xi]$ denote the real algebra of quaternion polynomials in the single 
variable $\xi$. In the present context, we regard a polynomial $\A(\xi)\in 
\hh[\xi]$ as the corresponding polynomial curve $\A:\rr\to\hh$, and use 
the notations
\[
\left(\sum_{r=0}^m \A_r\xi^r \right) \left(\sum_{s=0}^n \B_s\xi^s \right) 
= \sum_{r=0}^{m+n} \left( \sum_{s=0}^r \A_s\B_{r-s} \right) \xi^r \,,
\]
\[
\left( \sum_{r=0}^m \A_r\xi^r \right)^{\!\!*} = \sum_{r=0}^m \A_r^*\xi^r 
\]
for the multiplication and conjugation operations in $\hh[\xi]$. These
notations are also used for the subalgebra $\cc[\xi]$ of $\hh[\xi]$.

To obtain PH curves that are regular on all of $\rr$, the hodograph 
\eqref{hodo} should have no zeros in $\rr$. To this end, only elements
$\A(\xi)$ of the sets
\begin{align*}
\widetilde{\cc[\xi]}&:=\{a+b\,\i \in \cc[\xi] : 
a,b \in \rr[\xi], \textstyle{\gcd_\rr}(a,b)=1\},\\
\widetilde{\hh[\xi]}&:=\{u+v\,\i+p\,\j+q\,\k \in \hh[\xi] : 
u,v,p,q \in \rr[\xi], \textstyle{\gcd_\rr}(u,v,p,q)=1\}
\end{align*}
of complex and quaternion polynomials that have coprime real components 
are considered. Also, let $\langle\ ,\ \rangle$ denote the standard 
Euclidean scalar product of $\rr^4\cong\hh$. Then, for any two quaternions 
$$
\X=x_0+x_1\i+x_2\j+x_3\k
\quad \mbox{and} \quad 
\Y=y_0+y_1\i+y_2\j+y_3\k ,
$$
the quantity
\begin{equation}
\label{fou}
\frac{\langle \X,\Y\rangle}{\langle \Y, \Y\rangle} =
\frac{x_0y_0+x_1y_1+x_2y_2+x_3y_3}{y_0^2+y_1^2+y_2^2+y_3^2}
\end{equation}
is called the \emph{normalized component} of $\X$ along $\Y$ --- it indicates 
the oriented length of the orthogonal projection of $\X$ onto $\Y$, measured 
as a multiple of $|\Y|$. The normalized component offers a ``geometrical'' 
interpretation of the condition \eqref{rrmfeqn} that characterizes the RRMF 
curves, as follows.

\begin{lem}
\label{indicatrix}
For a quaternion polynomial of the form (\ref{A}) with coprime real 
components, the normalized component of $\A'\i$ along $\A$, and of $\A'$ 
along $\A\,\i$, can be computed as follows
\begin{equation} 
\label{fourier}
\frac{\langle\A'\i,\A\rangle}{\langle\A,\A\rangle} = 
-\frac{\langle\A',\A\,\i\rangle}{\langle\A\,\i,\A\,\i\rangle} =
-\frac{\langle\A',\A\,\i\rangle}{\langle\A,\A\rangle} =
-\frac{v'u-u'v-q'p+p'q}{u^2+v^2+p^2+q^2} \,.
\end{equation}
\end{lem}

\begin{proof}
Multiplying $\A'$ by $\i$ on the right gives $\A'\i=-v'+u'\i+q'\j-p'\k$,
and hence
$$
\langle\A'\i,\A\rangle = -v'u+u'v+q'p-p'q=-(v'u-u'v-q'p+p'q) \,.
$$
The result \eqref{fourier} then follows directly from the fact that
multiplying by a unit quaternion amounts to an orthogonal transformation 
of $\rr^4\cong\hh$, and hence
\[
\langle\A\,\i,\A\,\i\rangle = \langle\A,\A\rangle
\qquad \mbox{and} \qquad
\langle\A'\i,\A\rangle = -\langle\A',\A\,\i\rangle \,. \qedhere
\]
\end{proof}

The preceding result motivates the following definition.

\begin{defi}
\label{indicator}
For a quaternion polynomial of the form \eqref{A} with coprime real 
components, the function specified by the normalized component of $\A'\i$ 
along $\A$, i.e., the real function
\begin{equation}
\label{Findicatrix}
\frac{\langle\A'\i,\A\rangle}{\langle\A,\A\rangle} 
\,=\, -\frac{v'u-u'v-q'p+p'q }{u^2+v^2+p^2+q^2} \,,
\end{equation}
will be called the \emph{rotation indicatrix} of $\A(\xi)$. From equation 
\eqref{omega1} it is evident that, to obtain an RMF, the rate of instantaneous 
rotation that must be applied to the ERF vectors $(\e_2,\e_3)$ of the PH 
curve defined by $\r'=\A\,\i\,\A^*$ is twice the rotation indicatrix of $\A$.
For notational purposes, we also define the rotation indicatrix of the 
polynomial $\A=0$ to be $0$.
\end{defi}

The RRMF curves can be characterized in terms of the rotation indicatrix
\eqref{Findicatrix} of the generating quaternion polynomial \eqref{A} as 
follows.

\begin{teo}
\label{Fourier}
For a PH curve $\mathbf{r} (\xi)$ generated by the quaternion polynomial 
\eqref{A} with coprime real components, the following statements are 
equivalent:
\begin{enumerate}
\item $\mathbf{r}(\xi)$ is an RRMF curve.
\item There exists a complex polynomial $\gamma(\xi)=a(\xi)+b(\xi)\,\i \in 
\cc[\xi]$ with coprime real components, such that $\A(\xi)$ and $\gamma(\xi)$ 
have the same rotation indicatrix, i.e., 
\begin{equation}
\label{chiave}
\frac{\langle \A'\i, \A\rangle}{\langle \A, \A\rangle} \equiv
\frac{\langle \gamma'\i, \gamma\rangle}{\langle \gamma, \gamma\rangle} .
\end{equation}
\end{enumerate}
\end{teo}
\begin{proof}
The proof is just a restatement of Theorem~\ref{Han}, obtained by applying 
Lemma~\ref{indicatrix} to $\A(\xi)$ and $\gamma(\xi)$.
\end{proof}

Some key properties of the rotation indicatrices of quaternion polynomials 
are now derived. The first result expresses the rotation  indicatrix of a 
product of two quaternion polynomials in terms of the rotation indicatrices
of the individual polynomials and their components.

\begin{pro}
\label{formulaprodotto}
If $\A(\xi)$, $\B(\xi)$ are quaternion polynomials with coprime real 
components and $\alpha(\xi)$, $\beta(\xi)$ are the complex polynomials 
such that $\A(\xi)=\alpha(\xi)+\beta(\xi)\,\j$, the rotation indicatrix 
of the product $\B\A$ has the form
\[
\frac{\langle (\B\A)'\i, \B \A \rangle}{|\B\A|^2} =
\frac{|\alpha|^2-|\beta|^2}{|\alpha|^2+|\beta|^2}\,
\frac{\langle\B'\i,\B\rangle}{|\B|^2} -
\frac{2|\alpha||\beta|}{|\alpha|^2+|\beta|^2}\,
\frac{\langle\B'\frac{\alpha\beta}{|\alpha\beta|}\k,\B\rangle}{|\B|^2} + 
\frac{\langle\A'\i,\A\rangle}{|\A|^2} \,.
\]
\end{pro}

\begin{proof}
A direct computation shows that
\begin{eqnarray*}
\frac{\langle (\B\A)'\i, \B \A \rangle}{|\B\A|^2} %\\
\!\! &=& \!\! \frac{\langle \B' \A\i, \B \A \rangle}{|\B\A|^2} +
\frac{\langle \B\A'\i, \B \A \rangle}{|\B\A|^2} \\
\!\! &=& \!\! \frac{\langle \B' \A\i, \B \A \rangle}{|\B|^2|\A|^2} +
\frac{\langle |\B|\frac{\B}{|\B|}\A'\i, |\B|\frac{\B}{|\B|} \A \rangle}
{|\B|^2|\A|^2} \\
\!\! &=& \!\! \frac{\langle \B' \A\i, \B \A \rangle}{|\B|^2|\A|^2} +
\frac{|\B|^2\langle \frac{\B}{|\B|}\A'\i, \frac{\B}{|\B|} \A \rangle}
{|\B|^2|\A|^2} \\
\!\! &=& \!\! \frac{\langle \B' \A\i, \B \A \rangle}{|\B|^2|\A|^2} +
\frac{\langle \frac{\B}{|\B|}\A'\i, \frac{\B}{|\B|} \A \rangle}{|\A|^2}.
\end{eqnarray*}
Now since multiplication by a unit quaternion corresponds to an orthogonal 
transformation of $\rr^4\cong\hh$, we have
\begin{equation}
\label{equazione generale}
\frac{\langle(\B\A)'\i,\B\A\rangle}{|\B\A|^2} = 
\frac{\langle\B'\A\i,\B\A\rangle}{|\B|^2|\A|^2}  
+ \frac{\langle\A'\i,\A \rangle}{|\A|^2} \,.
\end{equation}
Writing $\A(\xi)=\alpha(\xi)+\beta(\xi)\,\j$, we obtain
\begin{eqnarray*}
&& \frac{\langle \B'\A\i,  \B\A \rangle}{|\B|^2|\A|^2} 
= \frac{\langle\B'(\alpha+\beta\j)\i,\B(\alpha+\beta\j)\rangle}{|\B|^2|\A|^2} 
= \frac{\langle\B'\i(\alpha-\beta\j),\B(\alpha+\beta\j)\rangle}{|\B|^2|\A|^2} \\
&&= \frac{\langle\B'\i\alpha,\B\alpha\rangle}{|\B|^2|\A|^2}
- \frac{\langle\B'\i\beta\j,\B\alpha\rangle}{|\B|^2|\A|^2} 
+ \frac{\langle\B'\i\alpha,\B\beta\j\rangle}{|\B|^2|\A|^2} 
- \frac{\langle\B'\i\beta\j,\B\beta\j\rangle}{|\B|^2|\A|^2} \\
&&= |\alpha|^2\frac{\langle\B'\i\frac{\alpha}{|\alpha|},  
\B\frac{\alpha}{|\alpha|}\rangle}{|\B|^2|\A|^2}
- |\alpha||\beta|\frac{\langle\B'\i\frac{\beta }{|\beta|}\j,  
\B\frac{\alpha}{|\alpha|}\rangle}{|\B|^2|\A|^2} \\
&& +\; |\alpha||\beta|\frac{\langle \B'\i\frac{\alpha}{|\alpha|},  
\B\frac{\beta}{|\beta|}\j\rangle}{|\B|^2|\A|^2} 
- |\beta|^2\frac{\langle\B'\i\frac{\beta\j}{|\beta|},  
\B\frac{\beta \j}{|\beta|} \rangle}{|\B|^2|\A|^2} \\
&&= \frac{|\alpha|^2-|\beta|^2}{|\alpha|^2+|\beta|^2}
\frac{\langle\B'\i,\B\rangle}{|\B|^2}
+ \frac{|\alpha||\beta|}{|\alpha|^2+|\beta|^2}
\frac{\langle\B'\i\frac{\beta}{|\beta|},\B\frac{\alpha}{|\alpha|}\j\rangle
+ \langle\B'\i\frac{\alpha}{|\alpha|},\B\frac{\beta }{|\beta|}\j\rangle}
{|\B|^2} \,.
\end{eqnarray*}
Note that $\frac{\alpha}{|\alpha|}$ and $\frac{\beta}{|\beta|}\j$ are unit 
quaternions, so their inverses are simply their conjugates. Multiplying by 
unit quaternions, and noting that $\alpha\,\j=\j\,\alpha^*$ and $\beta\,\j 
=\j\,\beta^*$, we have
\[
\langle\B'\i\frac{ \beta }{|\beta|},\B\frac{\alpha}{|\alpha|}\j\rangle 
= -\langle\B'\i\frac{\beta}{|\beta|}\j,\B\frac{\alpha}{|\alpha|}\rangle 
= -\langle\B'\k\frac{\beta^*}{|\beta|}\frac{\alpha^*}{|\alpha|},\B\rangle 
= -\langle\B'\frac{\beta}{|\beta|}\frac{\alpha}{|\alpha|}\k,\B\rangle \,,
\]
\[
\langle\B'\i\frac{\alpha}{|\alpha|},\B\frac{\beta}{|\beta|}\j \rangle
= -\langle\B'\i\frac{\alpha}{|\alpha|}\j,\B\frac{\beta }{|\beta|} \rangle
= -\langle\B'\k\frac{\alpha^*}{|\alpha|}\frac{\beta^*}{|\beta|},\B\rangle 
= -\langle\B'\frac{\alpha}{|\alpha|}\frac{ \beta }{|\beta|}\k,\B\rangle \,.
\]
Hence, we have shown that
\begin{eqnarray*}
\frac{\langle\B'\A\i,\B\A\rangle}{|\B|^2|\A|^2} \,=\,
\frac{|\alpha|^2-|\beta|^2}{|\alpha|^2+|\beta|^2}\,
\frac{\langle\B'\i,\B\rangle}{|\B|^2} \,-\, 
\frac{2|\alpha||\beta|}{|\alpha|^2+|\beta|^2}\,
\frac{\langle\B'\frac{\alpha\beta}{|\alpha\beta|}\k,\B\rangle}{|\B|^2} \,.
\end{eqnarray*}
The result follows directly from this last equality and equation 
\eqref{equazione generale}.
\end{proof}

The following result\footnote{This result was previously stated, in 
somewhat different terms, in Lemma~2.1 of \cite{farouki13b}.} expresses 
the rotation indicatrix of the product of a quaternion polynomial and a 
complex polynomial in terms of their individual rotation indicatrices.

\begin{cor}
\label{corollario}
For a given complex polynomial $\alpha(\xi)\in\widetilde{\cc[\xi]}$ and 
quaternion polynomial $\B(\xi)\in\widetilde{\hh[\xi]}$, the rotation 
indicatrix of the product $\B\alpha$ is the sum of the rotation indicatrices 
of the polynomials $\B$ and $\alpha$, i.e.,
$$
\frac{\langle(\B\alpha)'\i,\B\alpha\rangle}{|\B\alpha|^2} \equiv
\frac{\langle\B'\i,\B\rangle}{|\B|^2} +
\frac{\langle\alpha'\i,\alpha\rangle}{|\alpha|^2}.
$$
\end{cor}
\begin{proof}
The proof is a direct consequence of Proposition~\ref{formulaprodotto}. 
\end{proof}

Before proceeding, we mention one further consequence of Corollary~\ref
{corollario}.

\begin{cor}\label{bar}
For any $\delta(\xi)\in\widetilde{\cc[\xi]}$ the rotation indicatrices of
$\delta(\xi)$ and $\delta^*(\xi)$ differ only in sign, i.e.,
\[
 \frac{\langle{\delta^*}'\i,\delta^*\rangle}{|\delta^*|^2}
= -\frac{\langle\delta'\i,\delta\rangle}{|\delta|^2} \,.
\]
\end{cor}

\begin{proof} From Corollary~\ref{corollario} with $\B=\delta$ and $\alpha=
\delta^*$ we have
\[
0 = \frac{\langle(|\delta|^2)'\i,|\delta|^2\rangle}{|\delta|^4} = 
\frac{\langle (\delta\delta^*)'\i,\delta\delta^*\rangle}{|\delta|^2|\delta|^2} 
= \frac{\langle\delta'\i,\delta\rangle}{|\delta|^2} + 
\frac{\langle{\delta^*}'\i,\delta^*\rangle}{|\delta^*|^2} . \qedhere
\]
\end{proof}

\section{Reduction to vanishing indicatrix case}
\label{sec:reduction}

In this section, the study of the set of quaternion polynomials that 
generate RRMF curves is reduced to the study of the set of quaternion 
polynomials whose rotation indicatrix is identically zero.

\begin{defi}
For any $\gamma(\xi)\in \widetilde{\cc[\xi]}$, let
$$
\mathscr F_\gamma = 
\left\{ \A \in \widetilde{\hh[\xi]} \,:\, 
\frac{\langle \A'\i, \A \rangle}{|\A|^2} =
\frac{\langle \gamma'\i, \gamma \rangle}{|\gamma|^2}\right\}
$$
be the set of quaternion polynomials whose rotation indicatrix coincides
with that of $\gamma$. Moreover, let
$$
\mathscr F_0 = \left\{ \A \in \widetilde{\hh[\xi]} \,:\, 
\langle \A'\i, \A \rangle=0\right\}
$$
be the set of quaternion polynomials with vanishing rotation indicatrix.
\end{defi}

The following definition will be useful in the study of $\mathscr F_\gamma$.

\begin{defi}
For any quaternion polynomials $\A_1(\xi),\ldots,\A_n(\xi) \in \hh[\xi]$, their
\emph{greatest common right divisor} is defined as the unique monic 
polynomial $\C(\xi)\in \hh[\xi]$ having the following properties:
\bi
\item $\C(\xi)$ divides $\A_1(\xi),\ldots,\A_n(\xi)$ on the right
\item if $\B(\xi)$ divides $\A_1(\xi),\ldots,\A_n(\xi)$ on the right, then 
$\B(\xi)$ divides $\C(\xi)$ on the right.
\ei
The polynomial $\C$ will be denoted by $\gcd_\hh(\A_1,\ldots,\A_n)$.
\end{defi}

While the definition is well--posed in general (see \cite{damiano10} and 
references therein), we will only use it in the following special case.

\begin{oss}
If $\A(\xi)=\alpha(\xi)+\beta(\xi)\,\j=\alpha(\xi)+\j\,\beta^*(\xi)$ and if 
$\gamma(\xi) \in \cc[\xi]$, then 
$\gcd_\hh(\A,\gamma)=\gcd_\cc(\alpha,\beta^*,\gamma)$. 
\end{oss}

The next result reduces the study of the set $\mathscr F_\gamma$ to the 
study of $\mathscr F_0$. 

\begin{teo} 
\label{finale}
For any complex $\gamma(\xi)\in\widetilde{\cc[\xi]}$, we have
\[
\mathscr{F}_0{\gamma} \cap \widetilde{\hh[\xi]} 
\subseteq \mathscr{F}_\gamma \,.
\]
Moreover, 
\[
\mathscr{F}_\gamma = \left\{\A \in \widetilde{\hh[\xi]} \,:\, 
\A\,\gamma^*\,|\textstyle{\gcd_\hh}(\A,\gamma)|^{-2} \in \mathscr{F}_0\right\} .
\]
\end{teo}

\begin{proof}
If $\B\in \mathscr F_0$, then 
\[
\frac{\langle \B'\i, \B \rangle}{|\B|^2}=0
\]
by the definition of $\mathscr F_0$, and from Corollary~\ref{corollario} 
we have 
\[
\frac{\langle (\B\gamma)'\i, (\B \gamma) \rangle}{|\B\gamma|^2}=  
\frac{\langle \gamma'\i, \gamma \rangle}{|\gamma|^2}.
\]
Hence, $\B{\gamma} \in \mathscr F_\gamma$ provided the real components of 
$\B{\gamma}$ are still coprime.

Consider now the second statement. If $\A \in \mathscr{F}_\gamma$, then 
by Corollaries~\ref{corollario} and \ref{bar}, the rotation  indicatrix of 
$\C:=\A\gamma^*$ vanishes identically. The greatest real common divisor of 
the real components of $\C$ is $|\gcd_\hh(\A,\gamma)|^{2}$. Setting $\B:= 
\C\,|\gcd_\hh(\A,\gamma)|^{-2}$, we have $\B\in \mathscr{F}_0$ since the real 
components of $\B$ are coprime and the rotation indicatrix of $\B$ vanishes 
identically by Corollary~\ref{corollario}. Conversely, if $\A \in \widetilde
{\hh[\xi]}$ and $\B:=\A\,\gamma^*\,|\gcd_\hh(\A,\gamma)|^{-2}$ belongs to 
$\mathscr{F}_0$, then 
\[
\A = \B\,{\gamma^*}^{-1}|\textstyle{\gcd_\hh}(\A,\gamma)|^{2} = 
\B\,\gamma\,\displaystyle\frac{\left|\gcd_\hh(\A,\gamma)\right|^{2}}{|\gamma|^2}
\]
has the same rotation indicatrix as $\gamma$ by Corollary~\ref{corollario}. 
Hence, $\A \in \mathscr{F}_\gamma$.
\end{proof}

Recall that a polynomial $\A(\xi)\in\mathscr{F}_0$ generates an RRMF curve 
for which the Euler--Rodrigues frame \eqref{erf} is rotation--minimizing.
The next result sheds some light on the expression \eqref{rmf} for the
rational rotation--minimizing frame. 

\begin{pro}
\label{prop:great}
A quaternion polynomial $\A\in\mathscr{F}_\gamma$ generates a PH curve with 
a rational RMF that coincides with the Euler--Rodrigues frame \eqref{erf} of 
the curve generated by the polynomial
\[
\A\,\gamma^*\,|\textstyle{\gcd_\hh}(\A,\gamma)|^{-2} \,,
\]
which belongs to $\mathscr{F}_0$.
\end{pro}

\begin{proof}
The rational RMF of the curve generated by $\A\,\gamma^*\,|\gcd_\hh(\A,
\gamma)|^{-2}$ is its Euler--Rodrigues frame. Since $|\gcd_\hh(\A,\gamma)|$ 
cancels out in the expression of the ERF, this frame coincides with the 
expression~\eqref{rmf} for the rational RMF of the curve generated by 
$\A(\xi)$, where $\B(\xi):=\A(\xi)\gamma^*(\xi)$.
\end{proof}

Theorem~\ref{finale} permits a complete characterization of the set of 
quaternion polynomials that generate RRMF curves, which can be specialized 
to the case of curves with primitive hodographs (see Definition~\ref
{primitive}). The following definition and remark will be useful in
formulating this characterization.

\begin{defi}
\label{core}
For $\A(\xi) \in \hh[\xi]$, let $\alpha(\xi),\beta(\xi) \in \cc[\xi]$ be 
such that $\A(\xi)=\alpha(\xi)+\beta(\xi)\,\j$ and let 
\[
\chi(\xi)=\textstyle{\gcd_\cc}(\alpha(\xi),\beta^*(\xi)) \,,
\] 
i.e., $\chi$ is the highest--degree monic complex polynomial that divides 
$\A$ on the right. Then the polynomial $\A(\xi)\chi(\xi)^{-1}$ is called the 
\emph{core} of $\A(\xi)$.
\end{defi}

\begin{oss}
The core of $\A(\xi)$ coincides with $\A(\xi)$ if and only if $\r'(\xi)=
\A(\xi)\,\i\,\A^*(\xi)$ is a primitive hodograph.
\end{oss}

The promised characterization of the complete space of RRMF curves is 
formulated in the following theorem.

\begin{teo}
\label{G}
The set 
\[\mathscr{F} = \left\{\A(\xi) \in \widetilde{\hh[\xi]}: 
\A(\xi) \mathrm{\ generates\ an\ RRMF\ curve}\right\}\] 
can be characterized as follows:
\[
\mathscr{F} = \left\{\C(\xi)\,{\delta(\xi)} \,:\, \C(\xi) 
\mathrm{\ is\ the\ core\ of\ an\ element\ of\ } \mathscr{F}_0
\mathrm{\ and\ } \delta(\xi)\in \widetilde{\cc[\xi]}
\right\}.
\]
Consequently, a quaternion polynomial $\C(\xi) \in \mathscr{F}$ will
generate a primitive hodograph $\C(\xi)\, \i\, \C^*(\xi)$ if and only 
if it is the core of an element of $\mathscr{F}_0$.
\end{teo}

\begin{proof}
By Theorem \ref{finale}, we observe that 
\[
\mathscr{F} = \bigcup_{\gamma\in\widetilde{\cc[\xi]}} \mathscr{F}_\gamma 
= \left\{\A \in \widetilde{\hh[\xi]} \,:\, \exists\ \gamma \in 
\widetilde{\cc[\xi]} \mathrm{\ s.t.\ }\A\,\gamma^*\,|\textstyle{\gcd_\hh}
(\A,\gamma)|^{-2} \in \mathscr{F}_0\right\} \,.
\] 
Now if $\C(\xi)$ is the core of a polynomial $\A(\xi)\in\widetilde{\hh[\xi]}$, 
then
\[
\A(\xi) = \C(\xi) {\delta(\xi)}
\]
for some $\delta(\xi)\in \widetilde{\cc[\xi]}$. The existence of a 
$\gamma \in \widetilde{\cc[\xi]}$ such that 
\[
\mathscr{F}_0 \ni \A\,\gamma^*\,|\textstyle{\gcd_\hh}(\A,\gamma)|^{-2} 
= \C\, \delta\, \gamma^* |\textstyle{\gcd_\cc}(\delta,\gamma)|^{-2}
\]
implies that $\C$ is the core of an element of $\mathscr{F}_0$. Conversely, 
if $\C$ is the core of $\B \in\mathscr{F}_0$, then
\[
\mathscr{F}_0\ni \B(\xi) = \C(\xi) \mu(\xi)
\]
for some $\mu \in \widetilde{\cc[\xi]}$, whence $\C\in\mathscr{F}_{\mu^*}$ 
by Corollaries~\ref{corollario} and~\ref{bar}. Consequently, for every 
$\delta(\xi)\in \widetilde{\cc[\xi]}$ the product $\C(\xi)\delta(\xi)$ has 
the same rotation indicatrix as $\nu:=\mu^*\delta\,|\gcd_\cc(\mu,\delta)|
^{-2}$, and is therefore an element of $\mathscr{F}_\nu \subset \mathscr{F}$.
\end{proof}

The study of the set $\mathscr{F}$ has thus been reduced to the study of 
$\mathscr{F}_0$, which we undertake in the following section.

\section{Polynomials with vanishing indicatrix}
\label{sec:vanishingindicatrix}

Two characterizations of the polynomials $\A(\xi) \in \mathscr{F}_0$ 
are presented below. The first characterization is expressed in terms of 
the complex polynomials $\alpha(\xi),\beta(\xi) \in \cc[\xi]$ such that 
$\A(\xi)=\alpha(\xi)+\beta(\xi)\,\j$.

\begin{pro}
\label{caso piano}
Let $\A(\xi)\in\widetilde{\hh[\xi]}$ and let $\alpha(\xi),
\beta(\xi)\in\cc[\xi]$ be such that $\A(\xi)=\alpha(\xi)+\beta(\xi)\,\j$. 
Then $\A(\xi)\in\mathscr{F}_0$ if and only if
\[
\langle \alpha'\i,\alpha\rangle = \langle \beta'\i,\beta\rangle \,.
\]
\end{pro}

\begin{proof}
By direct computation,
\begin{align*}
\langle \A'\i,\A\rangle &= \langle \alpha'\i,\alpha\rangle 
+ \langle \alpha'\i,\beta \j\rangle 
+ \langle \beta' \j\i,\alpha\rangle 
+ \langle \beta' \j\i,\beta \j\rangle \\
&= \langle \alpha'\i,\alpha\rangle 
- \langle \alpha'\k,\beta \rangle 
- \langle \beta'\k,\alpha\rangle
- \langle \beta'\i,\beta\rangle,
\end{align*}
where $\langle\alpha'\k,\beta\rangle=0=\langle\beta'\k,\alpha\rangle$, 
since $\cc\k$ and $\cc$ span mutually orthogonal planes in $\rr^4\cong\hh$.
\end{proof}

The second characterization of $\mathscr{F}_0$ identifies the polynomials 
of degree $n$ belonging to $\mathscr{F}_0$ by means of $2n-1$ real equations.  
In order to fully justify the notations used below, we identify any 
polynomial $\A_n\xi^n+\cdots+\A_1\xi+\A_0$ with a series $\sum_{m\in\nn} 
\A_m\xi^m$ whose coefficients vanish for all $m>n$.

\begin{teo}
Let $\A(\xi)=\A_n\xi^n + \cdots + \A_1\xi + \A_0$ be a quaternion 
polynomial of degree $n$ with coprime real components. Then $\A \in 
\mathscr{F}_0$ if and only if all the real numbers defined by
\[
c_m^{(n)} := \sum_{k=0}^m (k+1)\left\langle \A_{m-k}, \A_{k+1}\i\right\rangle ,
\quad m=0,\ldots,2n-2
\]
vanish. Moreover, if the symbols $c_m^{(n-1)}$ denote the analogous 
expressions for the polynomial $\A_{n-1}\xi^{n-1}+\cdots+\A_1\xi+\A_0$, then 
the following equalities hold:
\ba
\label{inductivestep}
c^{(n)}_m \!\! &=& \!\!
c^{(n-1)}_m \,, \quad m=0,\ldots,n-2 \,, \nonumber \\
c^{(n)}_m \!\! &=& \!\!
c^{(n-1)}_m + (2n-m-1)\langle{\cal A}_{m+1-n},{\cal A}_n\i\rangle \,,
\quad m=n-1,\ldots,2n-4 \,, \qquad \\
c^{(n)}_m \!\! &=& \!\!
(2n-m-1)\langle{\cal A}_{m+1-n},{\cal A}_n\i\rangle \,,
\quad m=2n-3,2n-2 \,, \nonumber \\
c^{(n)}_m \!\! &=& \!\! 0 \,,
\quad m\geq2n-1 \,. \nonumber
\ea
\end{teo}

\begin{proof}
By definition $\A \in \mathscr{F}_0 \Longleftrightarrow \langle\A,\A'\i\rangle 
\equiv 0$. Now by direct computation,
\[
\A'(\xi)\,\i \,=\, n\A_n\i\,\xi^{n-1} + \cdots + \A_1\i
\]
and
\[\
\langle\A,\A'\i\rangle \,=\, 
c_{2n-1}^{(n)}\xi^{2n-1} + \cdots + c_{1}^{(n)}\xi + c_{0}^{(n)} \,,
\]
where 
\[
c_m^{(n)} \,:=\, 
\sum_{k=0}^m (k+1)\left\langle \A_{m-k}, \A_{k+1}\i\right\rangle \,.
\]
For all $m\leq n-2$ this expression does not involve the coefficient $\A_n$, 
and the equality $c_{m}^{(n)}=c_{m}^{(n-1)}$ immediately follows. For $n-1 
\leq m \leq 2n-2$, we have 
\begin{align*}
c_{m}^{(n)} &= c_{m}^{(n-1)} + n\left\langle \A_{m-n+1},\A_n\i\right\rangle 
+ (m-n+1)\left\langle \A_{n}, \A_{m-n+1}\i\right\rangle\\
&= c_{m}^{(n-1)} + (2n -m -1) \left\langle \A_{m-n+1}, \A_n\i\right\rangle \,,
\end{align*}
since $\langle\A_{n},\A_{m-n+1}\i\rangle=-\langle\A_n\i,\A_{m-n+1}\rangle$. 
Finally, $c_{2n-1}^{(n)}=n\langle\A_n,\A_n\i\rangle=0$, $c_{2n-3}^{(n-1)}=
(n-1)\langle\A_{n-1},\A_{n-1}\i\rangle=0$, and $c_m^{(n-1)}=0$ for $m>2n-3$.
\end{proof}

The following remark, which is a direct consequence of Proposition~\ref
{formulaprodotto}, allows us to work up to multiplication by a quaternion.

\begin{oss}
\label{uptoconstant}
$\A(\xi)\in\mathscr{F}_0 \Longleftrightarrow \C\A(\xi)\in\mathscr{F}_0$
for all $\A(\xi)\in\hh[\xi]$ and nonzero $\C \in \hh$.
\end{oss}

We are now ready to determine a special subclass of $\mathscr{F}_0$.

\begin{cor}\label{special}
Let $\A(\xi)$ be a quaternion polynomial of degree $n$ with coprime real 
components. If there exists a nonzero $\C \in \hh$ and a unit vector
$\u \perp\i$ such that $\A(\xi)=\C(\A_n\xi^n+\cdots+\A_1\xi+\A_0)$, with 
$\A_0,\ldots,\A_n \in \rr+\rr \u$, then $\A(\xi) \in \mathscr{F}_0$.
\end{cor}

\begin{proof}
In view of Remark~\ref{uptoconstant}, the result is proved if we can verify
that, for each unit vector $\u \perp\i$, any polynomial $\A_n\xi^n+\cdots 
+\A_1\xi+\A_0$ with $\A_0,\ldots,\A_n \in \rr+\rr\u$ belongs to $\mathscr{F}
_0$. We prove this by induction.

The statement is clearly true for $n=0$, since by inspection all constants 
belong to $\mathscr{F}_0$. Now suppose it is true for $n=k-1$, so the $c_m^
{(k-1)}$ corresponding to $\A_0,\ldots,\A_{k-1}$ all vanish. We will show 
that it is also true for $n=k$, i.e., the $c_m^{(k)}$ corresponding to $\A_0,
\ldots,\A_k$ all vanish as well. According to the inductive hypothesis and 
formulae \eqref{inductivestep}, we need only show that $\A_k\i$ is orthogonal 
to $\A_0,\ldots,\A_{k-1} \in \rr+\rr \u$. This is indeed the case, since 
$\A_k\i$ belongs to the plane $\rr\i+\rr\u\i$, which is orthogonal to 
$\rr+\rr\u$.
\end{proof}

Inspired by the last corollary, we give the following definition.

\begin{defi}\label{trivial}
A polynomial $\A(\xi)\in\mathscr{F}_0$ is called \emph{trivial} if a 
nonzero $\C\in\hh$ and a unit vector $\u$ with $\u\perp\i$ exist, such 
that $\A(\xi)=\C\tilde\A(\xi)$ for some polynomial $\tilde\A(\xi)$ whose 
coefficients all lie in $\rr+\rr\u$.
\end{defi}

The set of trivial elements of $\mathscr{F}_0$ is studied in the following 
theorem and the subsequent remark. First, we introduce the notation
\[
\mathscr{F}_0^{(n)} := \{\A \in \mathscr{F}_0 : \deg(\A)=n\} \,.
\]

\begin{teo}
\label{trivialhasnofactors}
Let $\A(\xi)$ be a trivial element of $\mathscr{F}_0$. Then $\A$ coincides 
with its own core. As a consequence, the hodograph $\r'(\xi)=\A(\xi)\,\i\,
\A^*(\xi)$ is primitive.
\end{teo}

\begin{proof}
Let $\A(\xi)\in\mathscr{F}_0^{(n)}$ be trivial. By Definition~\ref
{trivial}, $\A(\xi)=\C\tilde \A(\xi)$ where $\tilde \A(\xi)= \tilde\A_n
\xi^n+\cdots+\tilde\A_0$ has all of its coefficients $\tilde\A_0,\ldots,
\tilde\A_n$ in $\rr+\rr\u$ for some unit vector $\u$ with $\u\perp\i$. 
Suppose, by contradiction, that $\A(\xi)$ is divisible to the right by 
a complex polynomial $\gamma(\xi)$ of degree $m\geq1$, and consequently 
$\tilde\A(\xi)$ is divisible (on the right) by $\xi-\alpha_0$ for some 
$\alpha_0 \in \cc$. Then one of the quaternion factorizations 
\[
\tilde \A(\xi) = \tilde \A_n(\xi - \Q_1) \ldots (\xi - \Q_n)
\]
of $\tilde\A(\xi)$ has $\Q_n = \alpha_0 \in \cc$. From the properties 
of quaternion factorizations studied in \cite{gentili08a,gentili08b} the 
fact that $\tilde\A_0,\ldots,\tilde\A_n \in \rr+\rr\u$ along with the fact 
that $\Q_n \in \cc$ implies that: (i) $\Q_n \in \rr$; or (ii) $\Q_{n-1}= 
\Q_n^*$. But in case (i) the real components of $\A(\xi)$ would share 
the common factor $\xi-\Q_n$, which is excluded by the definition of 
$\mathscr{F}_0$. In case (ii), they would share the common factor $\xi^2 
-\xi(\Q_n+\Q_n^*)+|\Q_n|^2$, which is again a contradiction.
\end{proof}

By inspection of the formulae~\eqref{inductivestep} we obtain the following.

\begin{oss}
The set of trivial elements of $\mathscr{F}_0^{(n)}$ has real dimension 
$2n+5$ for all $n\geq 1$.
\end{oss}

\section{Identification of non-planar RRMF curves}
\label{sec:nonplanar}

The present section identifies all elements of $\mathscr{F}$ that generate
non--planar PH curves. It is natural to begin by studying $\mathscr{F}_0$. 
As part of the main result, we recover a property shown in \cite{choi02a}, 
namely, the simplest quaternion polynomials $\A(\xi)\in\mathscr{F}_0$ that 
generate non--planar RRMF curves are cubic.

\begin{teo}\label{teotrivial}
Let $\r(\xi)$ be a PH curve with hodograph $\r'(\xi)=\A(\xi)\,\i\,\A^*(\xi)$ 
for some polynomial $\A(\xi)\in\mathscr{F}_0$. Then $\r(\xi)$ is planar if 
and only if $\A(\xi)$ is trivial, which is always the case if $\deg(\A)\leq 
2$. Moreover, $\r(\xi)$ is a straight line if and only if $\deg(\A)=0$.
\end{teo}

\begin{proof}
Observe first that $\r(\xi)$ is planar if and only if the hodograph 
\eqref{hodo} ranges in a plane through the origin. Also, if $\tilde\A
(\xi)=\C\A(\xi)$ for nonzero $\C\in\hh$ and if $\A(\xi)\,\i\,\A^*(\xi)$ ranges 
in a plane $\Pi$ through the origin, then $\tilde\A(\xi)\,\i\,\tilde\A(\xi)$ 
ranges in the plane $C\,\Pi\,\C^*$ through the origin. Thus, we can argue
up to any (nonzero) constant quaternion factor. For any 
\[
\A(\xi) \,=\, \A_n\xi^n + \cdots + \A_1\xi + \A_0 
\,\in\, \mathscr{F}_0 \,, 
\] 
the fact that $\A(\xi)$ has coprime real components implies that $\A_0\neq0$. 
We may therefore assume, without loss of generality, that $\A_0=1$. Under 
this assumption, if we set
\begin{align*}
m &:= \max\{k : \A_0=1,\A_1,\ldots,\A_k \in \rr\} \\
M &:= \max\{k : \exists\ \u \mathrm{\ with\ } |\u|=1,
\u \perp\i\mathrm{\ s.t.\ } \A_0=1,\A_1,\ldots,\A_k \in \rr+\rr \u\}
\end{align*}
then $0\leq m \leq M \leq n$ and the theorem is equivalent to the following 
statements: $\A\,\i\,\A^*$ ranges in a plane through $0$ if and only if 
$M=n$; it ranges in a line through $0$ if and only if $m=M=n = 0$. We 
will prove both statements using the fact that, by direct computation,
\[
\A(\xi)\,\i\,\A(\xi)^* = \b_{2n}\xi^{2n}+\cdots+\b_1\xi+\b_0 \,, \quad
\b_l = \sum_{k=0}^l \A_k\i \A_{l-k}^* .
\]
Note that $\r'=\A\,\i\,\A^*$ ranges in a plane $\Pi$ through the origin 
if and only if $\r'$ and all its derivatives $\r^{(l)}$ for $l>1$ range in 
$\Pi$, and this is equivalent to stating that the pure vectors $\b_0,\ldots,
\b_{2n}$ all belong to $\Pi$.

\medskip\noindent
If $m=M=n$, then $\A_0,\ldots,\A_n \in \rr$. Since $\A$ has coprime real 
components, we deduce that $n=0$. 

\medskip\noindent
If $m<M=n$, then: (i) $\A_0=1,\ldots,\A_m \in \rr$; (ii) there exists a 
unit vector $\u \perp\i$ such that $\A_{m+1},\ldots,\A_n\in\rr+\rr\u$; 
and (iii) $\A_{m+1} \in (\rr+\rr\u)\setminus\rr$. By inspection, (i) implies 
that $\b_0,\ldots,\b_m\in\rr\i$, and (ii) implies that $\b_{m+1},\ldots, 
\b_{2n}$ all belong to the plane $\rr\i+\rr\u\i$. Moreover, 
\[
\b_{m+1} = \sum_{k=0}^{m+1} \A_k\i\A_{m+1-k}^* , 
\]
where all the summands belong to $\rr\i$ except $\A_0\i\A_{m+1}^*+
\A_{m+1}\i\A_0^*=2\A_{m+1}\i$, which is linearly independent of $\i$ by 
property (iii). Therefore, the span of $\b_0,\ldots,\b_{2n}$ is the plane 
$\Pi=\rr\i+\rr\u\i$. 

\medskip\noindent
If, on the other hand, $M < n$ we will prove that the span of $\b_0,\ldots,
\b_{2n}$ is the entire $3$-space $\rr\i+\rr\j+\rr\k$. By construction, the 
first coefficients $\A_0=1, \ldots, \A_m$ are real and there exists a $\u 
\perp\i$ such that $\A_{m+1}, \ldots, \A_M \in \rr+\rr \u$. Moreover, by 
hypothesis $\A_{M+1} \not \in \rr+\rr \u$. We now claim that $\A_{M+1}\perp\i$. 
Since $\A \in \mathscr{F}_0$ we have $c_k^{(n)}=0$ for all $k$. By applying 
the formulae~\eqref{inductivestep} several times, we conclude, in particular, 
that
\begin{align*}
0 = c_M^{(n)} = \ldots = c_M^{(M+1)} = 
c_M^{(M)} + \langle \A_0, \A_{M+1}\i \rangle = 
c_M^{(M)} - \langle \i, \A_{M+1} \rangle.
\end{align*}
The claim is thus equivalent to $c_M^{(M)}=0$, which is true since $\A_M\xi^M
+\cdots+\A_1\xi+\A_0$ belongs to $\mathscr{F}_0$ by Corollary~\ref{special}.

\medskip\noindent
By the claim, $\A_{M+1}$ has the form $x+y\,\u+z\,\u\,\i$ for some $x,y,z 
\in \rr$. We must have $z\neq0$, otherwise the maximality of $M$ would be 
contradicted. Moreover, $m<M$. Indeed, if $\A_0,\ldots,\A_M$ all lie 
in $\rr$ then $\A_0,\ldots,\A_{M+1}$ are all contained in the plane 
$\rr+\rr\,{\bf v}$, where
\[
{\bf v} := \frac{y\,\u+z\,\u\,\i}{\sqrt{y^2+z^2}} \,,
\]
and this again contradicts the maximality of $M$. As in the case $m<M=n$, 
we have $\b_0,\ldots,\b_m \in \rr\i$ while $\b_{m+1},\ldots,\b_M \in \rr\i+
\rr\u\i$ with $\b_{m+1}$ linearly independent of $\i$. Hence, the span of 
$\b_0,\ldots,\b_M$ is the entire plane $\rr\i +\rr\u\i$. Moreover, 
\[ 
\b_{M+1} = \sum_{k=0}^{M+1} \A_k\i \A_{M+1-k}^*
\] 
where all the terms of this sum belong to $\rr\i + \rr \u\i$, except for
$\A_0\i \A_{M+1}^*+\A_{M+1}\i \A_0^*=2(x\i-z\u+y\u\i)$. Since $z \neq 0$, 
we conclude that the span of $\b_0,\ldots,\b_{M+1}$ is the $3$--space of 
pure vectors $\rr\i+\rr\j+\rr\k$, as desired.

\medskip\noindent
Finally, consider the least degree $n$ for which the strict inequality 
$M < n$ is possible. We have seen that this implies (i) $0\leq m < M < n$; 
(ii) $\A_{m+1} \in (\rr+\rr\u)\setminus\rr$ with $\u \in \s, \u \perp\i$ 
and $\A_{M+1} \in (\rr+\rr\u+\rr\u\i)\setminus(\rr+\rr\u)$. By (i) we have 
$n\geq 2$. Moreover, $n\neq 2$ since $n=2$ and (ii) would imply that
\[
\A_1 \in  (\rr + \rr \u) \setminus \rr,\quad \A_{2} \in 
(\rr + \rr \u + \rr \u\i) \setminus (\rr + \rr \u)
\]
and this contradicts the condition
\[
0 = c_2^{(2)} = \langle \A_1, \A_2\i \rangle,
\]
which is necessary for $\A$ to belong to $\mathscr{F}_0$.
\end{proof}

The previous result yields a characterization of the planar curves 
among the set of all RRMF curves.

\begin{teo}\label{planarcharacterization}
Let $\A(\xi)\in\mathscr{F}$. Then the PH curve $\r(\xi)$ generated by
\eqref{hodo} is planar if and only if the core of $\A(\xi)$ is a trivial 
element of $\mathscr{F}_0$.
\end{teo}

\begin{proof}
If $\A(\xi)\in\mathscr{F}$, then by Theorem~\ref{G} there exists $\B(\xi) 
\in \mathscr{F}_0$ with the same core $\C(\xi)$ as $\A(\xi)$, i.e.,
\[
\A(\xi) = \C(\xi) \alpha(\xi) , \quad \B(\xi) = \C(\xi) \beta(\xi)
\]
for some monic $\alpha(\xi),\beta(\xi) \in \widetilde{\cc[\xi]}$. By direct 
inspection of \eqref{hodo}, the PH curve generated by $\A$ is planar if and
only if the curve generated by $\B$ is planar. By Theorem~\ref{teotrivial}, 
the latter is equivalent to saying that $\B$ is a trivial element of 
$\mathscr{F}_0$. By Theorem~\ref{trivialhasnofactors}, $\B=\C$.
\end{proof}

We conclude this section by drawing some conclusions from the last result, 
making use of Definition~\ref{primitive}.

\begin{teo}\label{thm:sum}
Consider the split of $\mathscr{F}$ specified by
\[ 
\mathscr{F} = \mathscr{P} \cup \mathscr{N} \,,
\]
where $\mathscr{P}$ and $\mathscr{N}$ are the sets of polynomials 
$\A(\xi)\in\mathscr{F}$ that generate planar and non--planar RRMF curves, 
respectively. Then
\[
\mathscr{P}=\mathscr{P}_0\,\widetilde{\cc[\xi]}
\]
where $\mathscr{P}_0=\mathscr{P}\cap \mathscr{F}_0=\{\A \in\mathscr{F}_0 : 
\A\mathrm{\ is\ trivial}\}$. Moreover,
\[
\mathscr{N} = \left\{ \A(\xi)\in\widetilde{\hh[\xi]} : \A(\xi)
\mathrm{\ has\ the\ same\ core\ as\ some\ } \B(\xi) \in \mathscr{N}_0 \right\}.
\]
where $\mathscr{N}_0=\mathscr{N}\cap\mathscr{F}_0=\{\A\in\mathscr{F}_0 : 
\A\mathrm{\ is\ not\ trivial}\}$. Finally, a polynomial $\A(\xi) \in 
\widetilde{\hh[\xi]}$ generates a non-planar curve with a primitive hodograph 
$\r'(\xi)$ via~\eqref{hodo} if and only if it is the core of an element 
of $\mathscr{N}_0$.
\end{teo}

\begin{proof}
The displayed expression for $\mathscr{P}$ is a restatement of 
Theorem~\ref{planarcharacterization}. The expression for $\mathscr{N}$,
and its specialization to the case of a primitive hodograph, follow 
immediately upon taking into account Theorem~\ref{G}.
\end{proof}

\section{Non-trivial polynomials in $\mathscr{F}_0$}
\label{sec:nontrivial}

In view of the importance of the non--trivial elements of $\mathscr{F}_0$,
highlighted in the previous section, we now study them in greater detail.
It is known that all elements of $\mathscr{F}_0^{(0)}$, $\mathscr{F}_0^{(1)}$, 
$\mathscr{F}_0^{(2)}$ are trivial. On the other hand, we show here that 
$\mathscr{F}_0^{(n)}$ admits non--trivial elements for any $n\geq 3$. 
We begin by deriving complete characterizations for $\mathscr{F}_0^{(3)}$ 
and $\mathscr{F}_0^{(4)}$.

\begin{teo}\label{degree3}
The non--trivial elements of $\mathscr{F}_0^{(3)}$ are those polynomials 
\[
\A(\xi) \,=\, \C\,(\A_3\xi^3+\A_2\xi^2+\A_1\xi+1)
\]
where $\C\in\hh$ is nonzero, and $\A_1, \A_2, \A_3 \in \hh$ are such that:
\begin{itemize}
\item the span of $1,\A_1,\A_2$ is $\rr+\rr\j + \rr\k$
\item the vector part of $\A_3$ is the pure vector, parallel to the vector 
product $(\A_1\i)\times(\A_2\i)$, whose component along $\i$ is $\frac 13
\langle \A_1,\A_2\i\rangle$.
\end{itemize}
\end{teo}

\begin{proof}
By Remark~\ref{uptoconstant}, we need only identify the polynomials 
$\A(\xi) = \A_3\xi^3 + \A_2\xi^2 + \A_1\xi + 1$ that are non--trivial elements 
of $\mathscr{F}_0^{(3)}$. According to equations \eqref{inductivestep}, such 
an $\A(\xi)$ belongs to $\mathscr{F}_0$ if and only if
\ba
\label{F03}
0 \!\! &=& \!\!c_{0}^{(3)} = c_{0}^{(2)} = c_{0}^{(1)} = 
\langle\A_0,\A_1\i\rangle = -\langle\i,\A_1\rangle \nonumber \\
0 \!\! &=& \!\!c_{1}^{(3)} = c_{1}^{(2)} = 
2\langle\A_0,\A_2\i\rangle = -2\langle\i,\A_2\rangle \nonumber \\
0 \!\! &=& \!\!c_{2}^{(3)} = c_{2}^{(2)} + 3 \langle\A_0,\A_3\i\rangle 
= \langle\A_1,\A_2\i\rangle - 3 \langle\i,\A_3\rangle \\
0 \!\! &=& \!\! c_{3}^{(3)} = 2 \langle\A_1,\A_3\i\rangle 
= -2 \langle\A_1\i,\A_3\rangle \nonumber \\
0 \!\! &=& \!\!c_{4}^{(3)} = \langle\A_{2},\A_3\i\rangle 
= - \langle\A_2\i,\A_3\rangle. \nonumber
\ea
The preceding equations are equivalent to the following conditions:
\begin{itemize}
\item[(i)\phantom{ii}] $\A_1,\A_2 \in \rr+\rr\j + \rr\k$
\item[(ii)\phantom{i}] the component of $\A_3$ along $\i$ is 
$\frac 13\langle\A_1,\A_2\i\rangle$
\item[(iii)] $\A_3$ is orthogonal to both $\A_1\i$ and $\A_2\i$
\end{itemize}
If $1,\A_1,\A_2$ span the entire space $\rr+\rr\j +\rr\k$, then $\A(\xi)$ is 
not trivial. Moreover, in this case the pure vectors $\A_1\i$ and $\A_2\i$ 
are linearly independent so that $\A_3 \perp \A_1\i,\A_2\i$ if and only if 
the vector part of $\A_3$ is parallel to $(\A_1\i)\times(\A_2\i)$.

\medskip\noindent
If, on the other hand, $1,\A_1,\A_2$ do not span the entire space $\rr+
\rr\j + \rr\k$, then their span is included in some plane $\rr+\rr \u$ for
some unit vector $\u$ with $\u\perp\i$. But then condition (ii) becomes 
$\A_3 \perp\i$. Along with condition (iii), this implies that $\A(\xi)$ is 
trivial.
\end{proof} 

\begin{ese}
The polynomial $-\frac{1}{3}\i\,\xi^3+\j\,\xi^2+\k\,\xi+1$ is a non--trivial 
element of $\mathscr{F}_0^{(3)}$.
\end{ese}

A completely analogous argument yields the following result.

\begin{teo}
A monic polynomial $\xi^3+\A_2\xi^2+\A_1\xi+\A_0$ is a non--trivial element 
of $\mathscr{F}_0^{(3)}$ if and only if the span of $1,\A_1,\A_2$ is $\rr+
\rr\j+\rr\k$ and the vector part of $\A_0$ is the unique vector parallel to 
the vector product $(\A_1\i)\times(\A_2\i)$ whose $\i$ component is equal 
to $-\frac 13\langle \A_1,\A_2\i\rangle$.
\end{teo}

A system of constraints on the Bernstein coefficients of cubic polynomials 
$\A(\xi)$, that identifies elements of $\mathscr{F}_0^{(3)}$ and is equivalent 
to the conditions (\ref{F03}), was previously derived in scalar form in 
\cite{choi02a}, and in quaternion form in \cite{farouki13}. Consider now 
the case of polynomials $\A(\xi)$ of degree $4$.

\begin{teo}\label{degree4}
The elements of $\mathscr{F}_0^{(4)}$ are those polynomials 
\[
\A(\xi) \,=\, \C(\A_4\xi^4 +\A_3\xi^3+\A_2\xi^2+\A_1\xi+1)
\]
where $\C\in\hh$ is nonzero, and $\A_1,\A_2,\A_3,\A_4\in\hh$ are such that:
\begin{itemize}
\item $\A_1,\A_2\in\rr+\rr\j+\rr\k$;
\item the component of $\A_3$ along $\i$ is $\frac13\langle\A_1,\A_2\i\rangle$;
\item $\A_4$ is orthogonal to $\A_2\i$ and $\A_3\i$, its component along $\i$ 
is $\frac 12\langle\A_1,\A_3\i\rangle$, and  $\langle\A_1,\A_4\i\rangle=\frac 
13\langle \A_2,\A_3\i\rangle$.
\end{itemize}
Moreover, $\A(\xi)$ is non--trivial if and only if one of the following 
conditions is satisfied:
\begin{enumerate}
\item the span of $1,\A_1,\A_2$ is $\rr+\rr\j + \rr\k$;
\item the span of $1,\A_1,\A_2$ is a plane $\rr+\rr\u$ for some unit vector  
$\u$ with $\u\perp\i$, and the span of $1,\A_1,\A_2,\A_3$ is $\rr+\rr\j+\rr\k$.
\end{enumerate}
\end{teo}

\begin{proof}
As in the proof of Theorem~\ref{degree3}, it suffices to consider $\A(\xi)= 
\A_3\xi^3+\A_2\xi^2+\A_1\xi+1$. By equations \eqref{inductivestep}, we have 
$\A(\xi)\in \mathscr{F}_0^{(4)}$ if and only if
\begin{align*}
&0=c_{0}^{(4)}=c_{0}^{(3)} = c_{0}^{(2)} = c_{0}^{(1)} = 
\langle \A_0,\A_1\i \rangle = -\langle \i,\A_1\rangle \\
&0=c_{1}^{(4)}=c_{1}^{(3)} = c_{1}^{(2)} = 
2\langle \A_0,\A_2\i \rangle = -2\langle \i,\A_2\rangle \\
&0=c_{2}^{(4)}=c_{2}^{(3)} = c_{2}^{(2)} + 3 \langle \A_0,\A_3\i\rangle 
= \langle \A_1, \A_2\i \rangle - 3 \langle \i,\A_3\rangle \\
&0=c_{3}^{(4)}= c_{3}^{(3)} + 4 \langle \A_0,\A_4\i\rangle
= 2 \langle \A_1,\A_3\i\rangle -4 \langle \i,\A_4\rangle \\
&0=c_{4}^{(4)}=c_{4}^{(3)} + 3 \langle \A_1,\A_4\i\rangle
= \langle \A_{2},\A_3\i\rangle - 3 \langle \A_1\i,\A_4\rangle \\
&0=c_{5}^{(4)} = 2 \langle \A_2,\A_4\i\rangle = -2 \langle \A_2\i,\A_4\rangle\\
&0=c_{6}^{(4)} = \langle \A_{3},\A_4\i\rangle = - \langle \A_3\i,\A_4\rangle.
\end{align*}
Moreover, if $1,\A_1,\A_2$ span the entire space $\rr+\rr\j+\rr\k$ then 
$\A(\xi)$ is not trivial. If, on the other hand, the span of $1,\A_1,\A_2$ 
is included in some plane $\rr+\rr\u$, where $\u$ is a unit vector with
$\u \perp\i$, then $0=c_{2}^{(4)}$ implies the $\A_3 \perp\i$. Under this 
assumption, either $1,\A_1,\A_2,\A_3$ span the entire space $\rr+\rr\j+\rr\k$ 
or $\A_3 \in \rr+\rr\u$ as well. In the former case, $\A(\xi)$ is clearly 
not trivial. In the latter case, $0=c_{3}^{(4)}$ implies that $\A_4 \perp\i$ 
and, along with $0=c_{4}^{(4)}=c_{5}^{(4)}=c_{6}^{(4)}$, this implies that 
$\A(\xi)$ is trivial.
\end{proof} 

Examples of both types of non--trivial elements described in Theorem~\ref
{degree4} are exhibited below.

\begin{ese}
The polynomial $(-1+\frac{1}{3}\k)\xi^4+(\frac{1}{3}\i+\j)\xi^3+ \k\,\xi^2+ 
\j\,\xi+1$ is a non--trivial element of $\mathscr{F}_0^{(4)}$.
\end{ese}

\begin{ese}
The polynomial $2\,\i\,\xi^4+4\,\k\,\xi^3+\j\,\xi+1$ is a non--trivial element 
of $\mathscr{F}_0^{(4)}$.
\end{ese}

We conclude with a result concerning polynomials of arbitrary degree.

\begin{teo}
For all $n\geq 3$, the set $\mathscr{F}_0^{(n)}$ contains non--trivial elements.
\end{teo}

The theorem is established by means of the following example.

\begin{ese}
For each $n\geq 3$, the polynomial $\A(\xi)=(n-2)\,\i\,\xi^n+n\,\k\,\xi^{n-1}
+\j\,\xi+1$ is a non--trivial element of $\mathscr{F}_0^{(n)}$. Since $\A'(\xi)
\,\i=-n(n-2)\,\xi^{n-1}+n(n-1)\,\j\,\xi^{n-2}-\k$, we have
\[
\langle \A'(\xi)\i,\A(\xi) \rangle \,=\, 
(-n(n-2)+n(n-1)-n)\,\xi^{n-1} \,\equiv\, 0 \,,
\]
and hence $\A(\xi) \in \mathscr{F}_0^{(n)}$. Moreover, $\A(\xi)$ is 
non--trivial, since its constant term is $1$ and its leading coefficient 
is not orthogonal to $\i$.
\end{ese}

\section{Examples with non-vanishing indicatrix}
\label{sec:exm}

To complement the examples of RRMF curves generated by elements of
$\mathscr{F}_0$ in the previous section, the present section presents 
in greater detail examples of some non--planar RRMF curves generated by 
polynomials $\A(\xi)\in\mathscr{F}\setminus\mathscr{F}_0$. These examples 
show that the characterization of RRMF curves developed herein accommodates 
not only the generic case, in which \eqref{rrmfeqn} is satisfied with
$\deg(u^2+v^2+p^2+q^2)=\deg(a^2+b^2)$, but also cases where this does 
not hold. 

The chronological development of solutions to 
\eqref{rrmfeqn} may be summarized as follows. Choi and Han \cite{choi02a} 
first identified a family of degree 7 PH curves satisfying \eqref{rrmfeqn} 
with $\deg(u^2+v^2+p^2+q^2)=6$ and $\deg(a^2+b^2)=0$, for which the ERF is 
itself an RMF, i.e., the rotation \eqref{f23} is not required. Subsequently,
Han \cite{han08} demonstrated that no true spatial PH cubics can satisfy 
\eqref{rrmfeqn}. A family of spatial PH quintics satisfying \eqref{rrmfeqn} 
with $\deg(u^2+v^2+p^2+q^2)=\deg(a^2+b^2)=4$ was then identified in \cite
{farouki09}, and a much--simplified characterization of these quintic RRMF 
curves was developed in \cite{farouki10a}. Furthermore, a characterization 
of RRMF curves of \emph{any} degree, that satisfy \eqref{rrmfeqn} with 
$\deg(u^2+v^2+p^2+q^2)=\deg(a^2+b^2)$, was formulated as a polynomial 
divisibility condition in \cite{farouki10b}. Namely, the condition \eqref
{rrmfeqn} is satisfied if and only if the polynomials 
\ba
\rho \!\! &=& \!\! (up'-u'p+vq'-v'q)^2 + (uq'-u'q-vp'+v'p)^2 \,,  
\nonumber \\
\eta \!\! &=& \!\! (uu'+vv'+pp'+qq')^2 + (uv'-u'v-pq'+p'q)^2 \,,
\nonumber
\ea
are both divisible\footnote{Observe that $\rho+\eta=(u^2+v^2+p^2+q^2)\,
(u'^2+v'^2+p'^2+q'^2)$, so divisibility of either $\rho$ or $\eta$ by 
$\sigma$ implies divisibility of the other.} by $\sigma=u^2+v^2+p^2+q^2$. 

A family of RRMF quintics satisfying \eqref{rrmfeqn} with 
$\deg(u^2+v^2+p^2+q^2)=4$ and $\deg(a^2+b^2)=2$ was identified in \cite
{farouki12b}. Although it was stated in \cite{farouki12b} that solutions 
to \eqref{rrmfeqn} with $\deg(u^2+v^2+p^2+q^2)<\deg(a^2+b^2)$ are not 
possible, a family of RRMF quintics satisfying \eqref{rrmfeqn} with 
$\deg(u^2+v^2+p^2+q^2)=4$ and $\deg(a^2+b^2)=6$ has recently been
identified by Cheng and Sakkalis \cite{cheng16}. The following examples 
illustrate the existence of quintic RRMF curves that are proper space 
curves, and satisfy \eqref{rrmfeqn} with $\deg(a^2+b^2)$ less than, equal 
to, and greater than $\deg(u^2+v^2+p^2+q^2)=4$.

\begin{ese}
Consider the hodograph $\r'(\xi)=(x'(\xi),y'(\xi),z'(\xi))$ defined by 
\eqref{hodo}, where the quaternion polynomial \eqref{A} has the components
\[
u(\xi) = 21\,\xi^2+21\,\xi-142 \,, \;\;
v(\xi) = -\,21\,\xi-63 \,,
\]
\[
p(\xi) = 42\,\xi-34 \,, \;\;
q(\xi) = -\,42\,\xi+94 \,.
\]
Substituting these polynomials into \eqref{hodo} gives
\ba
x'(\xi) \!\! &=& \!\! 
441\,\xi^4+882\,\xi^3-8610\,\xi^2+7434\,\xi+14141 \,, 
\nonumber \\
y'(\xi) \!\! &=& \!\! -\,1764\,\xi^3+420\,\xi^2+12012\,\xi-22412 \,,
\nonumber \\
z'(\xi) \!\! &=& \!\! -\,1764\,\xi^3+1428\,\xi^2+14700\,\xi-21500 \,, 
\nonumber
\ea
Since $\gcd_\rr(x'(\xi),y'(\xi),z'(\xi))=1$ this is a primitive hodograph,
satisfying the Pythagorean condition \eqref{pythag} with parametric speed
\[
\sigma(\xi) = 21(21\,\xi^2+126\,\xi+325)(\xi^2-4\,\xi+5) \,.
\]
Note that $\r(\xi)$ it a true space curve, since $(\r'(\xi)\times\r''(\xi))
\cdot\r'''(\xi)\not\equiv0$, and the RRMF condition \eqref{rrmfeqn} is 
satisfied by polynomials $a(\xi)$, $b(\xi)$ with $\deg(a^2+b^2)=2$, namely
\[
a(\xi) = \xi-2 \,, \quad
b(\xi) = -1 \,.
\]
Note that $\gcd_\rr(uv'-u'v-pq'+p'q,u^2+v^2+p^2+q^2)=441\,\xi^2+2646\,\xi+
6825$, so a cancellation occurs on the left in \eqref{rrmfeqn}, and we have
\[
\frac{uv'-u'v-pq'+p'q}{u^2+v^2+p^2+q^2} \,=\, \frac{ab'-a'b}{a^2+b^2} 
\,=\, \frac{1}{\xi^2-4\,\xi+5} \,.
\]
\end{ese}

\begin{ese}
\label{example2}
Substituting the components 
\[
u(\xi) = 7\,\xi^2-22\,\xi+10 \,, \;\;
v(\xi) = -\,19\,\xi^2+14\,\xi \,, 
\]
\[
p(\xi) = -\,26\,\xi^2+16\,\xi \,, \;\;
q(\xi) = -\,2\,\xi^2+12\,\xi \,.
\]
for the quaternion polynomial \eqref{A} into \eqref{hodo} yields
\ba
x'(\xi) \!\! &=& \!\! 
-\,270\,\xi^4+40\,\xi^3+420\,\xi^2-440\,\xi+100 \,, 
\nonumber \\
y'(\xi) \!\! &=& \!\! 960\,\xi^4-1080\,\xi^3-120\,\xi^2+240\,\xi \,,
\nonumber \\
z'(\xi) \!\! &=& \!\! 440\,\xi^4-1880\,\xi^3+1560\,\xi^2-320\,\xi \,.
\nonumber
\ea
Since $\gcd_\rr(x'(\xi),y'(\xi),z'(\xi))=1$ this is a primitive hodograph,
satisfying the Pythagorean condition \eqref{pythag} with
\[
\sigma(\xi) = 1090\,\xi^4-1720\,\xi^3+1220\,\xi^2-440\,\xi+100 \,.
\]
One can verify that $(\r'(\xi)\times\r''(\xi))\cdot\r'''(\xi)\not\equiv0$, 
so $\r(\xi)$ it a true space curve. For this curve, the RRMF condition 
\eqref{rrmfeqn} is satisfied by polynomials $a(\xi)$, $b(\xi)$ with 
$\deg(a^2+b^2)=4$, namely
\[
a(\xi) = 27\,\xi^2-22\,\xi+10 \,, \quad
b(\xi) = -\,19\,\xi^2+14\,\xi \,.
\]
Since $\gcd_\rr(uv'-u'v-pq'+p'q,u^2+v^2+p^2+q^2)=\gcd_\rr(ab'-a'b,a^2+b^2)=1$, 
no cancellation occurs on the left or right in \eqref{rrmfeqn}, and we have
\[
\frac{uv'-u'v-pq'+p'q}{u^2+v^2+p^2+q^2} \,=\, \frac{ab'-a'b}{a^2+b^2} \,=\, 
\frac{4\,\xi^2-38\,\xi+14}{109\,\xi^4-172\,\xi^3+122\,\xi^2-44\,\xi+10} \,.
\]
Figure~\ref{fig:example} compares the variation of the Frenet frame and
the rational rotation--minimizing frame along the curve considered in this
example.
\end{ese}

\begin{figure}[htbp]
\centering
\epsfxsize=\textwidth \epsfbox{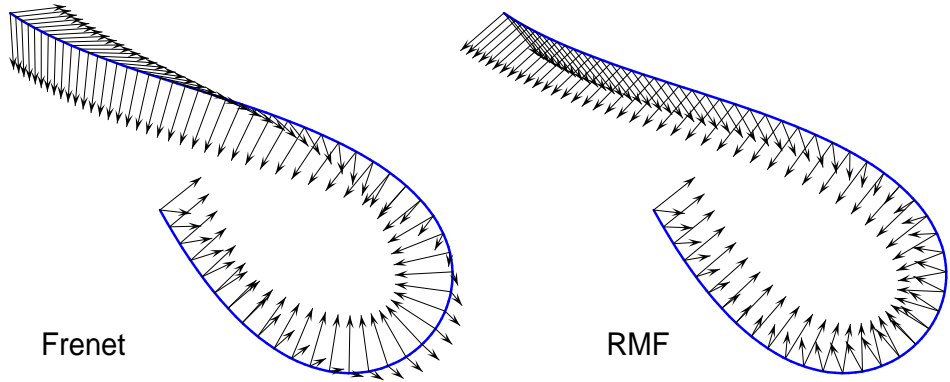}
\caption{Comparison of Frenet frame (left) and rotation--minimizing frame 
(right) along the RRMF curve in Example~\ref{example2}. For clarity, only 
the normal--plane vectors are shown (the RMF coincides with the Frenet frame 
at $\xi=0$).}
\label{fig:example}
\end{figure}

\begin{ese}
When the quaternion polynomial \eqref{A} has the components
\[
u(\xi) = 8\,\xi^2-35 \,, \;\;
v(\xi) = 16\,\xi^2-80\,\xi+90 \,, \;\;
p(\xi) = 3\sqrt{15} \,, \;\;
q(\xi) = -\,6\sqrt{15} \,,
\]
substituting into \eqref{hodo} gives
\ba
x'(\xi) \!\! &=& \!\! 
10\,(32\,\xi^4-256\,\xi^3+872\,\xi^2-1440\,\xi+865) \,, 
\nonumber \\
y'(\xi) \!\! &=& \!\! 480\sqrt{15}\,(-\,\xi+2) \,, \quad
z'(\xi) \,=\, 30\sqrt{15}\,(-8\,\xi^2+32\,\xi-29) \,, 
\nonumber
\ea
Since $\gcd_\rr(x'(\xi),y'(\xi),z'(\xi))=1$ this hodograph is primitive, and 
the Pythagorean condition \eqref{pythag} is satisfied with parametric speed
\[
\sigma(\xi) = 80(4\,\xi^2-16\,\xi+25)(\xi^2-4\,\xi+5) \,.
\]
Again $\r(\xi)$ is a true space curve, since $(\r'(\xi)\times\r''(\xi))
\cdot\r'''(\xi)\not\equiv0$. For this curve, condition \eqref{rrmfeqn} is 
satisfied by polynomials $a(\xi)$, $b(\xi)$ with $\deg(a^2+b^2)=6$, namely
\[
a(\xi) = 4\,\xi^3-24\,\xi^2+51\,\xi-38 \,, \quad
b(\xi) = -\,8\,\xi^2+32\,\xi-41 \,.
\]
Note that $\gcd_\rr(ab'-a'b,a^2+b^2)=4\,\xi^2-16\,\xi+25$, so a cancellation
occurs on the right in \eqref{rrmfeqn}, to give
\[
\frac{uv'-u'v-pq'+p'q}{u^2+v^2+p^2+q^2} \,=\, \frac{ab'-a'b}{a^2+b^2} \,=\,
\frac{8\,\xi^2-32\,\xi+35}{4\,\xi^4-32\,\xi^3+109\,\xi^2-180\,\xi+125} \,.
\]
\end{ese}

\section{Closure}
\label{sec:closure}

By introducing and exploiting the notions of the rotation indicatrix
and the core of a quaternion polynomial, a comprehensive theory of the 
entire space of polynomial curves that admit rational rotation--minimizing 
frames (RRMF curves) has been developed. This theory subsumes all the 
previously--known individual cases, and thus addresses a key open problem
in the understanding of RRMF curves identified in a recent survey paper
\cite{farouki16}. Moreover, the theory should prove useful in developing
practical algorithms for the construction of rational rotation--minimizing
rigid body motions, through the interpolation of discrete position and
orientation data \cite{farouki12,farouki13}.

Another important problem, on which the present theory can be brought
to bear, concerns the analysis of RRMF curves that satisfy condition 
\eqref{rrmfeqn} with $u^2+v^2+p^2+q^2$ and $a^2+b^2$ of both equal and 
unequal degree. Since the theory accommodates both cases, it may offer
a new path to the complete classification of possible cancellations
of non--constant factors common to the numerator and denominator on the
left or right of equation \eqref{rrmfeqn}. A detailed analysis of this 
problem is deferred to a future study.

\raggedright

\subsection*{Acknowledgements}

{\small
This work was supported by the following grants of the Italian Ministry 
of Education (MIUR): Futuro in Ricerca {\it DREAMS} (RBFR13FBI3); Futuro 
in Ricerca {\it Differential Geometry and Geometric Function Theory} 
(RBFR12W1AQ); and PRIN {\it Variet\`a reali e complesse: geometria, 
topologia e analisi armonica} (2010NNBZ78). It was also supported by 
the following research groups of the Istituto Nazionale di Alta Matematica 
(INdAM): Gruppo Nazionale per il Calcolo Scientifico (GNCS) and Gruppo 
Nazionale per le Strutture Algebriche, Geometriche e le loro Applicazioni 
(GNSAGA).
}

\def\AMC{{\it Appl.\ Math.\ Comput.\ }}
\def\AML{{\it Appl.\ Math.\ Lett.\ }}
\def\AMM{{\it Amer.\ Math.\ Monthly\ }}
\def\ACM{{\it Adv.\ Comp.\ Math.\ }}
\def\ACMTMS{{\it ACM Trans.\ Math.\ Software\ }}
\def\ACMTOG{{\it ACM Trans.\ Graphics\ }}
\def\BAMS{{\it Bull.\ Amer.\ Math.\ Soc.\ }}
\def\CA{{\it Comm.\ Alg.\ }}
\def\CAD{{\it Comput.\ Aided Design\ }}
\def\CAEJ{{\it Comput.\ Aided Eng. J.\ }}
\def\CAGD{{\it Comput.\ Aided Geom.\ Design }}
\def\CAVW{{\it Comput.\ Anim.\ Virt.\ Worlds }}
\def\CG{{\it Computers \& Graphics }}
\def\CGIP{{\it Comput.\ Graphics Image\ Proc.\ }}
\def\CMA{{\it Comput.\ Math.\ Applic.\ }}
\def\CV{{\it Complex Var.\ }}
\def\CVGIP{{\it Comput.\ Vision, Graphics, Image\ Proc.\ }}
\def\EJLA{{\it Electron.\ J.\ Lin.\ Alg.\ }}
\def\GM{{\it Graph.\ Models\ }}
\def\IBMJRD{{\it IBM J.\ Res.\ Develop.\ }}
\def\JCAM{{\it J.\ Comput.\ Appl.\ Math.\ }}
\def\IEEECGA{{\it IEEE Comput. Graph. Applic.\ }}
\def\IJAMT{{\it Int.\ J.\ Adv.\ Manuf.\ Tech.\ }}
\def\IJMS{{\it Int.\ J.\ Model.\ Sim.\ }}
\def\IJMTM{{\it Int.\ J.\ Mach.\ Tools Manuf.\ }}
\def\IJRR{{\it Int.\ J.\ Robot.\ Res.\ }}
\def\IMAJNA{{\it IMA J.\ Numer.\ Anal.\ }}
\def\IUJM{{\it Ind.\ Univ.\ J.\ Math.\ }}
\def\JMS{{\it J.\ Math.\ Sci.\ }}
\def\JC{{\it J.\ Complexity\ }}
\def\JSC{{\it J.\ Symb.\ Comput.\ }}
\def\MC{{\it Math.\ Comp.\ }}
\def\MMA{{\it Math.\ Model.\ Anal.\ }}
\def\MMAS{{\it Math.\ Methods\ Appl.\ Sci.\ }}
\def\MJM{{\it Milan J.\ Math.\ }}
\def\MMJ{{\it Mich.\ Math.\ J.\ }}
\def\MZ{{\it Math.\ Z.\ }}
\def\NA{{\it Numer.\ Algor.\ }}
\def\NMTMA{{\it Numer.\ Math.\ Theor.\ Meth.\ Appl.\ }}
\def\PAMS{{\it Proc.\ Amer.\ Math.\ Soc.\ }}
\def\SIAMJNA{{\it SIAM J.\ Numer.\ Anal.\ }}
\def\SIAMR{{\it SIAM Rev.\ }}
\def\TAMS{{\it Trans.\ Amer.\ Math.\ Soc.\ }}

\raggedright

\end{document}